\documentclass[oneside,12pt]{article}

\usepackage{amsmath}
\usepackage{amssymb}
\usepackage{pxfonts}
\usepackage{graphicx}
\usepackage{subcaption}
\usepackage{eucal}
\usepackage{mathrsfs}
\usepackage{theorem}
\usepackage{pifont}
\usepackage{color}
\usepackage{enumitem}
\usepackage[margin=2cm]{geometry}
\usepackage[backref=page]{hyperref}
\usepackage[normalem]{ulem}
\usepackage{dsfont}

\usepackage{pgfplots}

\usepackage{tocloft}

\definecolor{shadecolor}{rgb}{0.8,0.8,0.8}

\usepackage{amsmath}
\usepackage{amssymb}
\usepackage{graphicx}
\usepackage{eucal}
\usepackage{mathrsfs}
\usepackage{pifont}
\usepackage{color}
\usepackage{cjhebrew}

\usepackage[all]{xy}
\usepackage{tikz} % Required for drawing custom shapes
\usepackage{tkz-euclide}
\newcommand{\btkz}{\begin{tikzpicture}}
\newcommand{\etkz}{\end{tikzpicture}}

\newcommand{\brk}[1]{\left(#1\right)}          % \brk{.}     => (.)
\newcommand{\BRK}[1]{\left\{#1\right\}}        % \BRK{.}     => {.}
\newcommand{\Norm}[1]{\left\| #1 \right\|}     % \Norm{.}    => ||.||

\newcommand{\secref}[1]{Section~\ref{#1}}
\newcommand{\figref}[1]{Figure~\ref{#1}}
\newcommand{\thmref}[1]{Theorem~\ref{#1}}

\newcommand{\propref}[1]{Proposition~\ref{#1}}
\newcommand{\lemref}[1]{Lemma~\ref{#1}}

\newcommand{\beq}{\begin{equation}}
\newcommand{\eeq}{\end{equation}}
\newcommand{\bsplit}{\begin{split}}
\newcommand{\esplit}{\end{split}}
\newcommand{\baligned}{\begin{aligned}}
\newcommand{\ealigned}{\end{aligned}}

\newcommand{\Emph}[1]{{\slshape\bfseries #1}}  % \Emph{.}

\providecommand{\e}{\varepsilon}

\providecommand{\R}{\bbR}

\newcommand{\Textand}{\qquad\text{ and }\qquad}

\providecommand{\vp}{\varphi}

\newcommand{\sgn}{{\operatorname{sgn}}}

\newcommand{\calM}{{\mathcal M}}
\newcommand{\calN}{{\mathcal N}}

\newcommand{\frake}{\mathfrak{e}}

\newcommand{\frakg}{\mathfrak{g}}

\newcommand{\bbC}{{\mathbb C}}

\newcommand{\bbN}{{\mathbb N}}

\newcommand{\bbR}{{\mathbb R}}
\newcommand{\bbS}{{\mathbb S}}

\newcommand{\bbZ}{{\mathbb Z}}

\theoremheaderfont{\bfseries\rmfamily}
\newtheorem{theorem}{Theorem}[section]
\newtheorem{lemma}[theorem]{Lemma}
\newtheorem{proposition}[theorem]{Proposition}

\newtheorem{definition}[theorem]{Definition}

\theorembodyfont{\rmfamily}
\newtheorem{remark}[theorem]{Remark}

\newenvironment{proof}{{\flushleft \emph{Proof}:}}{\hfill\ding{110}}
\newenvironment{proof1}[1]{{\flushleft \emph{Proof #1}:}}{\hfill\ding{110}}

\usepackage{framed}
\usepackage{enumitem}
\definecolor{shadecolor}{rgb}{0.90,0.90,0.90}

\usepackage{xcolor} % Required for specifying colors by name
\definecolor{ocre}{RGB}{243,102,25}
\definecolor{darkocre}{RGB}{121,51,12}
\definecolor{lightocre}{RGB}{255,150,37}
\definecolor{verylightocre}{RGB}{255,204,50}
\definecolor{lightgray}{RGB}{200,200,200}
\definecolor{warmblue}{RGB}{51,102,153}
\definecolor{lightwarmblue}{RGB}{105,141,198}
\definecolor{sepia}{RGB}{112,66,20}

\newcommand{\W}{\Omega}
\newcommand{\M}{\calM}
\renewcommand{\Emph}[1]{{\bfseries #1}}
\newcommand{\n}{\calN}

\newcommand{\ind}{\operatorname{ind}}
\newcommand{\VolG}{\operatorname{dVol}_\g}
\newcommand{\ip}[1]{\langle #1 \rangle}
\newcommand{\g}{\frakg}
\newcommand{\nabg}{\nabla^\g}
\newcommand{\euc}{\frake}

\newcommand{\diam}{\operatorname{diam}}

\DeclareMathOperator*{\osc}{osc}

\newcommand{\Wiso}{W^{2,2}_\g(\M;\R^3)}
\newcommand{\II}{\operatorname{II}}
\newcommand{\Eb}{E_\text{bend}}

\newcommand{\Sph}{\bbS^2}
\newcommand{\dB}{\partial B}

\numberwithin{equation}{section}

\begin{document}

\title{Branch points and non-density for finite-bending isometric immersions of hyperbolic surfaces}
\author{Gilad Derfner\footnotemark[1] \and Raz Kupferman\thanks{Einstein Institute of Mathematics, Hebrew University of Jerusalem} \and Cy Maor\footnotemark[1]}
\date{}

\maketitle

\begin{abstract}
  The space of $W^{2,2}$-isometric immersions of a surface into $\R^3$ arises naturally in the variational theory of thin elastic sheets: it is precisely the finite-bending class, where the bending energy --- the $L^2$-norm of the second fundamental form --- is finite.
  For sheets with negative Gaussian curvature, previous work has identified branch points, where ``too many'' asymptotic directions meet --- or, equivalently, where the index of the Gauss map is not $-1$ --- as a potentially important mechanism in shape selection and pattern formation.
  Such branch points are precluded for $C^2$-isometric immersion.

  We show that this index-based notion of branch points extends to the full finite-bending class:
  Namely, for every $W^{2,2}$ isometric immersion of a negatively-curved surface, the index of the Gauss map is well-defined at every point, and the set of branch points is discrete.
  We further show that the index is stable under $W^{2,2}$-convergence, and thus, an isometric immersion with branch points cannot be approximated by $C^2$-isometric immersions.
  Conversely, we show that every negatively-curved metric locally admits $W^{2,2}$-isometric immersions (in fact, $C^{1,1}$) with branch points of arbitrary order.
  Consequently, $C^2$-isometric immersions are, in general, not dense among finite-bending ones, in stark contrast with the flat and positively curved cases.
\end{abstract}

\setcounter{tocdepth}{1}
\begingroup
\footnotesize
\tableofcontents
\endgroup
%%%%%%%%%%%%%%%%%%%%%%%%%%%%%%%%
\section{Introduction}

The study of isometric immersions of a Riemannian manifold $(\M,\g)$ in Euclidean space $(\R^d,\euc)$ is among the most fundamental subjects in Riemannian geometry.
It is well-known, since the work of Nash, that for a given metric, isometric immersions of different regularities can behave significantly differently \cite{Nas54}.
An extreme example of this is that smooth isometric immersions of the round sphere are unique up to rigid motions (Cohn-Vossen's theorem \cite{HC52}), whereas $C^1$ isometric immersions can approximate any short map (Nash--Kuiper \cite{Kui55}).
Of particular interest is the case of immersions of surfaces ($\dim\M=2$), into three-dimensional Euclidean space ($d=3$);
we will assume these dimensions from here onwards.

A particularly important class of isometric immersions are those of $W^{2,2}$ regularity, that is, maps $f:\M\to \R^3$ satisfying the isometry condition $df^Tdf =f^*\euc= \g$ almost everywhere, and having a square integrable second fundamental form ($\euc$ represents the Euclidean metric on $\R^3$).
The square of the $L^2$ norm of the second fundamental form is known as the \emph{bending energy} of the immersion:
\beq\label{eq:bending_energy}
\Eb(f) = \int_\M |\II|^2 \, \VolG = \int_\M |d\n|^2 \,\VolG,
\eeq
where $\II$ is the second fundamental form of $f$ in $\R^3$, $\n$ is its normal (the Gauss map), and the norms $|\cdot|$ are taken with respect to $\g$.
This energy is a central object in mathematical elasticity theory, as it arises as the leading order elastic energy for thin elastic plates \cite{FJM02b}.
It is also equivalent to the Willmore functional, which is well-studied in geometric analysis with many applications (e.g., \cite{Wil92}).
The space of $W^{2,2}$ isometric immersions, which we denote by $\Wiso$, is the natural space to study these functionals.

In certain cases, the space $\Wiso$ is rather well-behaved, in the sense that its elements have similar properties as smooth immersions:
\begin{enumerate}
  \item If $\M \subset \R^2$ is a Lipschitz domain with a piecewise continuously-differentiable boundary, and $\g$ is \emph{flat} (i.e., its Gaussian curvature $K_\g$ vanishes identically), then the maps in $\Wiso$ are developable, continuously differentiable, and approximable by smooth isometric immersions:
        \[
          \overline{C^\infty(\bar{\M};\R^3)\cap \Wiso }^{W^{2,2}} = \Wiso \subset C^1(\M;\R^3),
        \]
        where the notation $\overline{\,\cdot\,}^{W^{2,2}}$ represents the closure with respect to the $W^{2,2}$ topology \cite{Pak04,MP05,Hor11,Hor11b}.
        These results have been widely used in the elasticity of thin sheets (e.g., \cite{HNV14,Kup17,BPP25}).

  \item If $\g$ is \emph{elliptic}, i.e., $K_\g >c>0$, then, as shown in \cite{HV18}, finite-bending isometric immersions are smooth:
        \[
          \Wiso \subset C^\infty(\M;\R^3).
        \]
        In particular, any property of smooth isometric immersions (e.g., convexity) holds in this case.
\end{enumerate}
See also \cite{MS95,KMP26} for some other regularity properties of finite-bending immersions.
The goal of this paper is to show that for isometric immersions of \emph{hyperbolic} surfaces ($K_\g<0$) the behavior is significantly different from the flat and elliptic cases: smooth isometric immersions are \emph{not} dense in the space of isometric immersions having finite-bending, due to a fundamental difference in their behavior --- finite-bending isometric immersions may have \emph{branch points}.

\paragraph{Finite-bending isometric immersion of hyperbolic surfaces}
Consider a $C^2$ isometric immersion of a surface having negative Gaussian curvature.
Since the second fundamental form is negatively-defined (its determinant is $K_\g<0$ by Gauss' theorem), it has two null-directions; these are called \emph{asymptotic directions}.
On the other hand, there exist $C^{1,1}$ (or, equivalently, $W^{2,\infty}$) isometric immersions of hyperbolic surfaces that have, at certain points, more than two asymptotic directions;
such points are called \emph{branch points} \cite{GSSV16,SV21}.
See \cite[\S7]{GV11} and \cite{SV21} for various examples of isometric immersions of domain in the hyperbolic plane exhibiting branch points.\footnote{We also mention that $C^{1,1}$ isometric immersions of hyperbolic surfaces were recently studied from different perspectives \cite{CSW10,Li20}. See also \cite{Li26} for isometric immersions of lower regularity.}

An alternative and more useful definition of a branch point of an immersion $f$ is the following:
\begin{definition}
  Let $\g$ be a metric of negative Gaussian curvature, and let $f\in \Wiso$.
  A \Emph{Branch point} of $f$ is a point $p\in \M$ at which the \emph{index} (also referred to as the \emph{local degree}) of the Gauss map $\n$ at $p$ does not equal $-1$ (see \secref{sec:local_degree} for a definition of the index).
  If $p$ is a branch point whose index is $m$, we say that the order of the branch point is $1-m$.
\end{definition}

Note that for a $C^2$ isometric immersion of a hyperbolic surface, the Gauss map $\n$ is a $C^1$ orientation-reversing map, hence has index $-1$ at every point, and thus there are no branch points.
The order of the branch point at $p$ corresponds to half of the number of asymptotic rays emanating from $p$.
Note also that it is not a-priori clear that the index is well-defined at any $p\in\M$ for $f\in \Wiso$; this is part of Theorem~\ref{thm:main} below.

The occurrence of branch points in finite-bending isometric immersions is not just a mathematical phenomenon, but may be central for understanding some central questions in the behavior of physical systems:
Venkataramani and collaborators \cite{GV11,GSSV16,SV21} conjectured, and gave some numerical evidence, that when the surface is ``big enough", the occurrence of branch points can lower the bending energy, and is the source of fractal-like shapes exhibited by many biological and physical systems with underlying hyperbolic geometry (e.g., \cite{SRMSS02, SMS04, SRS07,KES07}), see \figref{fig:fractals}.
\begin{figure}
  \centering
  \subcaptionbox{}{\includegraphics[height=4.7cm]{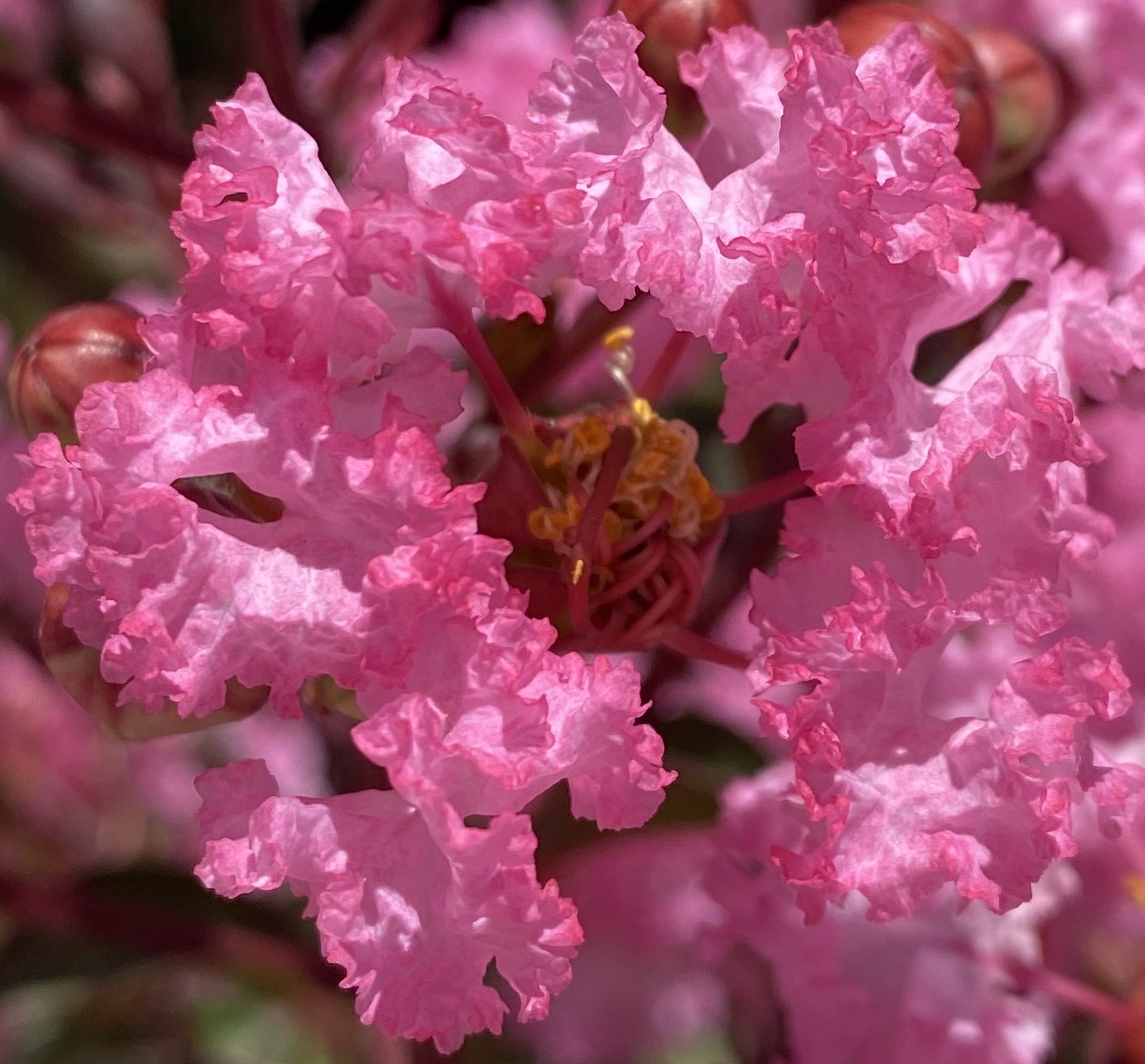}}
  \hfill
  \subcaptionbox{}{\includegraphics[height=4.7cm]{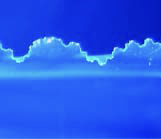}}
  \hfill
  \subcaptionbox{}{\includegraphics[height=4.7cm]{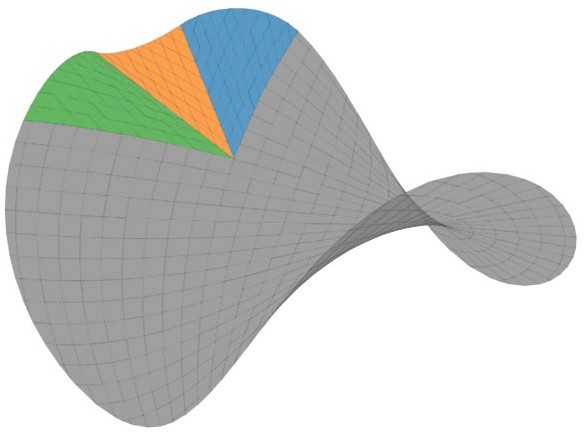}}
  \hfill
  \caption{Self-similar buckling patterns in hyperbolic sheets: (a) Flowers of a Lagerstroemia (photo by GD). (b) Torn nylon (from \cite{SMS04}). (c)
    A numerical calculation exemplifying the introduction of a branch point near the boundary of a surface (from \cite{SV21}), suggested as a potential mechanism for bending energy reduction.}
  \label{fig:fractals}
\end{figure}
Understanding this behavior of hyperbolic sheets is still a major open question in incompatible/non-Euclidean elasticity, and has been addressed from various approaches \cite{AB03,GV13,BK14, SV21}.
Since branch-points do not occur in $C^2$ isometric immersions, this conjecture implies that when the surface is in some sense sufficiently large  (possibly, when the total curvature is large enough), \beq\label{eq:energy_gap}
\inf_{W^{2,2}_\g} \Eb < \inf_{W^{2,2}_\g\cap C^2} \Eb.
\eeq
Since the bending energy is continuous with respect to the $W^{2,2}$ topology, this inequality can hold only if
\beq\label{eq:W22non_density}
\overline{C^2(\M;\R^3)\cap \Wiso }^{W^{2,2}} \subsetneq \Wiso,
\eeq
in stark contrast to the elliptic and flat cases described above.
In \cite[Theorem 3.22]{SV21} a non-density result of the type \eqref{eq:W22non_density} was proven, however under an additional assumption:
It was shown that if $(\M,\g)$ is part of the hyperbolic plane, and if $f\in W^{2,\infty}_\g(\M;\R^3)$ is an isometric immersion exhibiting a branch point, then $f$ cannot be approximated in the $W^{2,2}$ topology by $C^2$ isometric immersions that are \emph{uniformly bounded in $W^{2,\infty}$}.
In particular, they showed that
\[
  \overline{C^2(\M;\R^3)\cap W^{2,\infty}_\g(\M;\R^3) }^{W^{2,\infty}} \subsetneq W^{2,\infty}_\g(\M;\R^3).
\]

Our first main result describes the topology of the set of branch points for a general $W^{2,2}_\g$ immersion, and proves that whenever $W^{2,2}_\g$ contains an immersion with branch points, \eqref{eq:W22non_density} holds (without any additional assumptions):

\begin{theorem}\label{thm:main}
  Let $(\M,\g)$ be a two-dimensional Riemannian manifold without boundary having negative Gaussian curvature, and let $f:\M\to\R^3$ be a $W^{2,2}$ isometric immersion. Then
  \begin{enumerate}
    \item The index of the Gauss map of $f$ is well-defined at every point, and the set of branch points, i.e., points at which the index of the Gauss map does not equal $-1$, is closed and discrete.
    \item If $f$ admits a branch point, then $f$ cannot be approximated in the $W^{2,2}$ topology by $C^2$ isometric immersions, and thus \eqref{eq:W22non_density} holds.
  \end{enumerate}
\end{theorem}

We note that a somewhat similar notion of branch points exists for $W^{2,\infty}$ solutions of the hyperbolic Monge-Amp\`ere equation, $\det \nabla^2 u =-1$, in two dimensions \cite[Chapter 2]{Kir03}; these branch points are also discrete \cite[Theorem~2.20]{Kir03}.
However, the techniques of our result and Theorem~\ref{thm:main} differ significantly.

Note also that since our analysis is local, we do not have to assume anything about the completeness of $\M$,
which can be, for example, the interior of a compact manifold with boundary.
Consequently, the result also holds if one replaces $W^{2,2}$ with $W^{2,2}_\text{loc}$.

For the proof of the first part of the theorem, we apply a result of \cite{GHP19} concerning manifold-valued maps of finite distortion. We prove that $W^{2,2}$ isometric immersions of negatively-curved surfaces  have continuous Gauss maps.
This allows us to view the Gauss map locally as a map into $\R^2$, and apply a structure theorem for maps of integrable dilatation \cite[Thm.~1]{IS93}, which reduces the question to the discreteness of nonconstant analytic functions.

The proof of the second part is based, like \cite[Theorem 3.22]{SV21}, on a stability argument for the index. The main challenge when lowering the regularity from $W^{2,\infty}$ (i.e., $C^{1,1}$), as in \cite{SV21}, to $W^{2,2}$, is as follows:
A branch point is a point $p\in \M$ whose index does not equal $-1$, meaning that $\deg(\n,U,\n(p))\ne -1$, for some neighborhood $U$ of $p$ depending on the Gauss map $\n$.
Thus, we cannot simply use the continuity of the degree $\deg(\n_n,U,q)\to \deg(\n,U,q)$ with respect to uniform convergence to obtain the continuity of the index.
In \cite{SV21}, this is circumvented by showing that the uniform $W^{2,\infty}$ bounds imply that $U$ can be taken uniformly for that whole sequence.
In our case such a bound is not available, so a more delicate argument is needed to prove the stability of the index (as a side point, our argument removes the need to use the degree of VMO maps, as was done in \cite{SV21}).

\begin{remark}
  As mentioned above, the starting point of the proof is that a $\Wiso$ map of a negatively curved metric has a continuous normal (see \propref{prop:unif conv} below).
  In particular, this implies (see \cite[Lemmas~2.11, 2.13]{HV18}) that the image of the map is a $C^1$ surface, in the following sense:
  Let $f\in\Wiso$, then every point in $\M$ has a neighborhood $V$ such that $f(V)$ is the graph of a $C^1\cap W^{2,2}$ function $h:\W\to\R$, with $\W\subset\R^2$.
  We do not know whether $f$ itself is $C^1$ (like in the cases of flat or elliptic metrics).
\end{remark}

Our second main result shows that $W^{2,2}$ isometric immersions exhibiting branch points exist locally for every negatively-curved manifold:

%%%%%%%%%%%%%
\begin{theorem}
  \label{thm:generic}
  Let $(\M,\g)$ be a two-dimensional Riemannian manifold without boundary having negative Gaussian curvature.
  Then every point $p\in\M$ has an open neighborhood $U_p\subset\M$ and an immersion $f_p\in W^{2,\infty}_\g(U_p,\R^3)$ having a branch point at $p$ of arbitrary order.
\end{theorem}
%%%%%%%%%%%%%

This result generalizes similar constructions of isometric immersions with branch points for hyperbolic surfaces of constant curvature \cite{SV21} and for immersions with prescribed negative curvature \cite{PV26}.
The proof consists of two parts: (i) Proving the existence of $C^2$ isometric immersions of ``wedges" whose boundaries are geodesics which map to straight lines, which boils down to solving a characteristic hyperbolic system of semilinear PDEs. (ii) Gluing the wedges together and showing that the resulting immersion is of the correct regularity and contains a branch point.

\paragraph{Open questions and future directions.}
This paper establishes the index of the Gauss map as an important invariant for studying the space of $W^{2,2}$ isometric immersions of hyperbolic metrics.
One natural question that arises is: Are branch points the only obstructions for density, that is, does the closure of $C^2(\M;\R^3)\cap \Wiso $ consist exactly of the $W^{2,2}_\g$ maps having no branch points?
Furthermore, from an elasticity viewpoint, does the energy gap conjecture \eqref{eq:energy_gap} hold? Currently, this is still widely open even in the constant curvature case.

\paragraph{Structure of the paper.}
In Section~\ref{sec:setting} we set the notation and basic definitions.
In Section~\ref{sec:Gauss_map} we analyze the convergence of Gauss maps of $W^{2,2}$-convergent isometric immersions, and in particular show that they converge locally-uniformly.
In Section~\ref{sec:local_degree} we recall the index of a map, analyze its stability and conclude the proof of Theorem~\ref{thm:main}.
In \secref{sec:construction} we prove \thmref{thm:generic} whereby every point in a hyperbolic surface has a neighborhood that can be immersed in $\R^3$ with branch points.

\paragraph{Acknowledgements.}
CM thanks Robert Jerrard for insightful discussions on the subject, and in particular for pointing out the related result in \cite{Kir03}.
We are grateful to the user Volk on Stack Exchange
for suggesting an approach for the proof of \propref{prop:wedge immersion}.
RK was partially funded by ISF grant 560/22.
CM was partially funded by ISF grant 2304/24 and BSF grant 2022076.

%%%%%%%%%%%%%%%%%%%%%%%%%%%%%%%%
\section{Setting}\label{sec:setting}
Throughout this paper, we assume that $(\M,\g)$ is a smooth two-dimensional Riemannian manifold without boundary, which does not have to be complete.
We assume that $\g$ is a hyperbolic metric, i.e., with negative Gaussian curvature $K_\g\le c < 0$.
We denote by $\Wiso$ the space of finite-bending isometric immersions:
\[
  \Wiso := \BRK{f\in W^{2,2}(\M,\R^3) ~:~ f^*\euc = \g \,\,\text{almost everywhere}\,},
\]
where the isometry condition means that
\beq\label{eq:isometric_condition}
\g_p(u,v) =
\ip{df_p(u),df_p(v)}_\euc
\qquad \forall u,v\in T_p\M, \quad \text{for a.e. } p\in \M,
\eeq
and $\ip{\cdot,\cdot}_\euc$ denotes the inner-product in $\R^3$.
The norm of tensor field over $T\M$ and $T^*\M$ is denoted correspondingly by $|\cdot|_\g$. The norm of such tensor fields taking values in $\R^3$ is denoted by $|\cdot|_{\g,\euc}$. We will often omit the subscripts where there is no source of confusion.
The $L^2$ norm of functions (e.g., $f:\M\to \R^3$) or tensor fields (e.g., $df:\M\to T^*\M\otimes \R^3$) is always taken with respect to these pointwise norms, and the volume form $\VolG$ induced by $\g$.

Let $f\in \Wiso$ and let $\n:\M\to\Sph$ be the corresponding Gauss map, which is defined almost everywhere.  The Gauss map can be defined in a coordinate-free manner,
\beq
\n = \frac{*_\euc (df\wedge df)}{\VolG},
\label{eq:N}
\eeq
where $*_\euc : \Lambda^2 \R^3\to \Lambda^1\R^3\simeq \R^3$ is the Euclidean Hodge-dual operator.
If $\{x^1,x^2\}$ is a local coordinate system for $\M$, then within that chart
\[
  \n = \frac{\partial_1 f \times \partial_2 f}{\det(\g_{ij})},
\]
where $\g_{ij} = \g(\partial_i,\partial_j)$.

Denote by $\nabg$ the Levi-Civita connections on $T\M$, $T^*\M$ and their tensor products. The covariant Hessian of the configuration satisfies the following identity:
\beq
\nabg df (X,Y) = -\ip{df(X),d\n(Y)}_\euc\, \n.
\label{eq:Hess_f}
\eeq
Let $\{e_i\}$ be a local orthonormal frame field on $\M$. Since the range of $d\n$ is perpendicular to $\n$, and since $\{df(e_i)\}$ is an orthonormal frame field for $f(\M)$, it follows that
\[
  d\n(e_i) = \sum_j \ip{d\n(e_i),df(e_j)}_\euc\,df(e_j),
\]
hence
\beq
|d\n|_{\g,\euc}^2 = \sum_i |d\n(e_i)|^2_\euc = \sum_{i,j} \ip{d\n(e_i),df(e_j)}_\euc^2 = |\nabg df|^2_{\g,\euc}.
\label{eq:dN=Hess_f}
\eeq
It follows at once that $\Wiso$ consists of maps $f\in W^{2,1}(M;\R^3)$ satisfying \eqref{eq:isometric_condition} and having $\n\in W^{1,2}(\M;\Sph)$.
This justifies the name ``finite-bending isometries'' for $\Wiso$.

%%%%%%%%%%%%%%%%%%%%%%%%%%%%%%%%
\section{Gauss map analysis}
\label{sec:Gauss_map}
The main result of this section is the local-uniform convergence of normals of converging immersions in $\Wiso$ (Proposition~\ref{prop:unif conv}).
Before proving it, we start with a well-known proposition, the proof of which, however, is not easy to locate in the literature:

%%%%%%%%%%
\begin{proposition}[$W^{1,2}$ convergence of the Gauss Map]
  \label{prop:W22 continuity}
  Let $f_n\subset \Wiso$ and $f\in \Wiso$, such that $f_n\to f$ in $W^{2,2}$. Denote by $\n_n,\n$ the corresponding Gauss maps. Then $\n_n\to\n$ in $W^{1,2}(\M;\Sph)$. In particular, the bending energy $\Eb$ as defined in \eqref{eq:bending_energy} is continuous in $\Wiso$.
\end{proposition}
%%%%%%%%%%

%%%%%%%%%%
\begin{proof}
  Since $df_n \to df$ in $W^{1,2}(\M;\R^3)$ it follows from \eqref{eq:N} that $\n_n\to\n$ in $W^{1,p}(\M;\Sph)$ for every $p<2$; in particular, $\n_n\to\n$ in $L^2$ and $d\n_n\to d\n$ in measure.
  By \eqref{eq:dN=Hess_f},
  \[
    |d\n_n-d\n|^2\le2|d\n_n|^2+2|d\n|^2=2|\nabg df_n|^2+2|d\n|^2.
  \]
  Since $\|\nabg df_n\|_{L^2}\to\|\nabg df\|_{L^2}$, it follows by Pratt's lemma (see \cite[Thm.~A.10]{Rin18}) that $d\n_n\to d\n$ in $L^2$, thus $\n_n\to \n$ in $W^{1,2}$.
\end{proof}
%%%%%%%%%%

%%%%%%%%%%
\begin{proposition}[Uniform convergence of the Gauss map of hyperbolic immersions]
  \label{prop:unif conv}
  Let $f_n$, $\n_n$, $f$, $\n$ be defined as above.
  Then, $\n_n$ and  $\n$ are continuous, and $\n_n\to\n$ uniformly on every compact set $Z\subset\M$.
\end{proposition}
%%%%%%%%%%

%%%%%%%%%%
\begin{proof}
  By \cite[Proposition 5.4]{MM25}, the Gauss equation
  \beq\label{eq:Gauss_eq}
  \det(d\n \circ df^{-1}) = K_\g
  \eeq
  holds almost everywhere. Equivalently, $d\n$ viewed as bundle map $T\M\to T\Sph$ has a negative determinant (that is, its coordinate representation with respected to oriented coordinates in $\M$ and $\Sph$ has a negative determinant); by inverting the orientation of the sphere, $d\n$ has positive determinant almost-everywhere.
  In the terminology of \cite{GHP19}, $\n$ has
  \emph{finite distortion}.
  It follows from \cite[Theorem 4 and Corollary 5]{GHP19} that $\n$ is continuous.
  Moreover, let $Z\subset\M$ be compact; there exists an $R_0>0$ depending on $Z$ and $\n$,
  such that for every $x\in Z$ and every $0<r<R<R_0$:
  \beq
  \brk{\osc_{B(x,r)}\n}^2\le\frac{C}{\log(R/r)}\int_{B(x,R)}|d\n|^2\,\VolG=:I_{R,r},
  \label{eq:GHP}
  \eeq
  where $B(x,R)$ denotes a geodesic ball in $\M$, and for a set $A$,
  \[
    \osc_{A}\n = \diam(\n(A)).
  \]
  Similarly,  $\n_n$ is continuous for every $n\in\bbN$, and satisfies the same oscillation inequality.

  The radius $R_0$ is limited from above by properties of $\M$ and $Z$ and, using the notation of \cite{GHP19}, by the following radius
  \[
    R_2 = \sup\BRK{r > 0 ~\left|~ \int_{B(x,r)} |d\n|^2\,\VolG\right.  \le C_{\Sph},\,\,\, \forall x\in Z},
  \]
  where $C_{\Sph}>0$ is a number determined only by the geometry of the compact codomain (its radius of injectivity and its volume growth function). Since $L^2$-convergence implies equi-integrability, it follows that $R_2$, and hence $R_0$, can be chosen such that \eqref{eq:GHP} holds with $\n$ replaced by $\n_n$ for every $n$; denote the corresponding right-hand side by $I_{R,r}^n$.

  Let $R = R_0/2$ and let $\e>0$. Invoking again equi-integrability, there exists an $r\in(0,R)$ such that $I_{R,r}^n<\e^2$, namely,
  \[
    \osc_{B(x,r)}\n_n <\e,
  \]
  implying that $\n_n$ is equi-continuous in $Z$. Since the codomain is compact, it follows from the Arzela-Ascoli theorem that $\n_n$ has a subsequence that converges uniformly in $Z$. By \propref{prop:W22 continuity} and the uniqueness of the limit, this subsequence converges to $\n$. Since every subsequence of $\n_n$ has a further subsequence converging uniformly to $\n$, it follows that $\n_n\to\n$ uniformly in $Z$.
\end{proof}

%%%%%%%%%%%%%%%%%%%%%%%%%%%%%%%%%%%%%%%%%%
\section{Index analysis}
\label{sec:local_degree}

In this section we analyze the index of a map, and prove both parts of \thmref{thm:main}.
We start by recalling the definition of the index of a map; see  \cite[Definition~2.8]{FG95} for the index of a map between Euclidean domains. The only modification when considering maps between manifolds is using the definition of the degree of a map between manifolds (e.g., \cite[Definition~1.25]{FG95}).

\begin{definition}
  Let $M,N$ be two oriented manifolds of equal dimension.
  Let $u\in C(M;N)$ and let $p\in M$.
  Assume that there exists an open neighborhood $U$ of $p$, such that
  \[
    u^{-1}(\{u(p)\}) \cap \overline{U} = \{p\},
  \]
  that is, $p$ is an isolated point in the pre-image of $u(p)$.
  Then, we define the index of $u$ at $p$ by
  \[
    \ind_p(u) = \deg(u,U,u(p)),
  \]
  which is well-defined as it does not depend on the choice of $U$.
\end{definition}

\begin{proof1}{of Theorem~\ref{thm:main}, part 1}
  Let $f\in\Wiso$, let $\n$ be its Gauss map, and let $p\in\M$. By the continuity of $\n$, there exists an orientation preserving coordinate chart $\psi:U\to\R^2$ around $p$, such that $\n(U)$ is compactly contained in a single hemisphere.
  Denote by $P:\n(U)\to\R^2$ the orthogonal projection onto the plane defining the hemisphere (we identify this plane with $\R^2$ such that $P$ reverses orientation.
  Denote $\hat{\n}=P\circ\n\circ\psi^{-1}:\psi(U)\to P(\n(U))$, which is a map between bounded subsets of $\R^2$.
  Then, $\hat{\n}\in W^{1,2}(\psi(U);\R^2)$, $\det(d\hat{\n})>0$ a.e. (this follows from choosing $\psi$ to be orientation-preserving and $P$ to be orientation-reversing), and $|d\hat{\n}|^2/\det d\hat{\n}\in L^1(\W)$ because $d\hat{\n}\in L^2$ and $\det \hat{\n}$ is bounded away from zero, which follows from the Gauss equation \eqref{eq:Gauss_eq}.

  By the structure theorem \cite[Thm.~1]{IS93} there exists an orientation preserving homeomorphism $h\in W^{1,2}(\W,\psi(U))$ for some domain $\W\subset\R^2$, and an analytic function $\vp\in W^{1,2}(\W,\R^2)$, such that
  \[
    \hat{\n}=\vp\circ h^{-1}.
  \]

  Now, $h^{-1}(\psi(p))=:q$ has an open neighborhood $U_1$ such that $\vp^{-1}(\vp(q))\cap\overline{U_1}={q}$, otherwise $\vp$ would be constant by the identity theorem for analytic functions. Thus, since $h$ is a homeomorphism, $h(U_1)$ is an appropriate neighborhood for $\ind_{\psi(p)}\hat{\n}$ to be well-defined.
  Similarly, because $P,\psi^{-1}$ are homeomorphisms, $\psi^{-1}(h(U_1))$ is an appropriate neighborhood for $\ind_p\n$ to be well-defined.

  We next prove that the set of branch points in $U$ is discrete and closed in $U$.
  First, the set $A$ of points in $\W$ at which $\vp'=0$ is discrete and closed in $\W$, otherwise we would have $\vp'\equiv0$ again by the identity theorem.
  The set $B$ of points in $\W$ at which $\ind\vp\ne1$ is contained in $A$, because if $\vp'(z)\ne0$ then $\vp$ is an orientation preserving local homeomorphism in a neighborhood of $z$, so $\ind_z\vp=1$; see e.g., the proof of \cite[Thm. 3.35]{FG95}. In fact, $A=B$, for example by \cite[Thm. 2.20]{FG95} and the argument principle.
  Thus $B$ is discrete and closed in $\W$. Because $h$ is an orientation preserving homeomorphism, its index is defined everywhere and equals $1$, so by \cite[Thm. 2.10]{FG95} $h(B)$ is the set of points in $\psi(U)$ for which $\ind\hat{\n}\ne1$. Because $h$ is a homeomorphism, $h(B)$ is discrete and closed in $\psi(U)$. Similarly, $\psi^{-1}(h(B))$ is the set of points in $U$ for which $\ind\n\ne-1$ (because $P$ is orientation reversing), and it is discrete and closed in $U$.

  From this we conclude that the set of branch points is discrete and closed in $M$, because otherwise they would have an accumulation point in $M$, and then they would not be discrete and closed in a neighborhood of this accumulation point.
\end{proof1}

We proceed to show the stability of the index with respect to uniform convergence of local homeomorphisms (\propref{prop:unif conv preserves local deg}).
We start with a technical lemma:

%%%%%%%%%%%%%
\begin{lemma}
  \label{lem:unif_conv_k}
  Let $B_R(p)$ be a geodesic disc in $\M$.
  Let $u_n \in C(B_R(p); \Sph)$  be a sequence converging uniformly to $u\in C(B_R(p);\Sph)$, satisfying
  \[
    \ind_p(u) = d.
  \]
  Let $r\in(0,R)$ be such that the open neighborhood $B_r(p)$ of $p$ satisfies the conditions on $U$ in the definition of the index.
  Then, there exist $\e>0$, $\delta\in(0,r)$ and $N_0\in\bbN$ such that for all $n>N_0$,
  \begin{enumerate}[itemsep=0pt,label=(\alph*)]
    \item
          $B_\e(u(p)) \cap u_n(\dB_r(p)) = \emptyset$. That is, for $n$ large enough, the distance between $u(p)$ and the image of $\dB_r(p)$ under $u_n$ is at least $\e$.

    \item
          $u_n(B_\delta(p)) \subset B_{\e/2}(u(p))$.

    \item
          $\deg(u_n,B_r(p),u(p)) = d$.

    \item
          $u(p) \not\in u_n(\overline{B_r(p)} \setminus B_\delta(p))$.

  \end{enumerate}
\end{lemma}
%%%%%%%%%%%%%

\begin{figure}
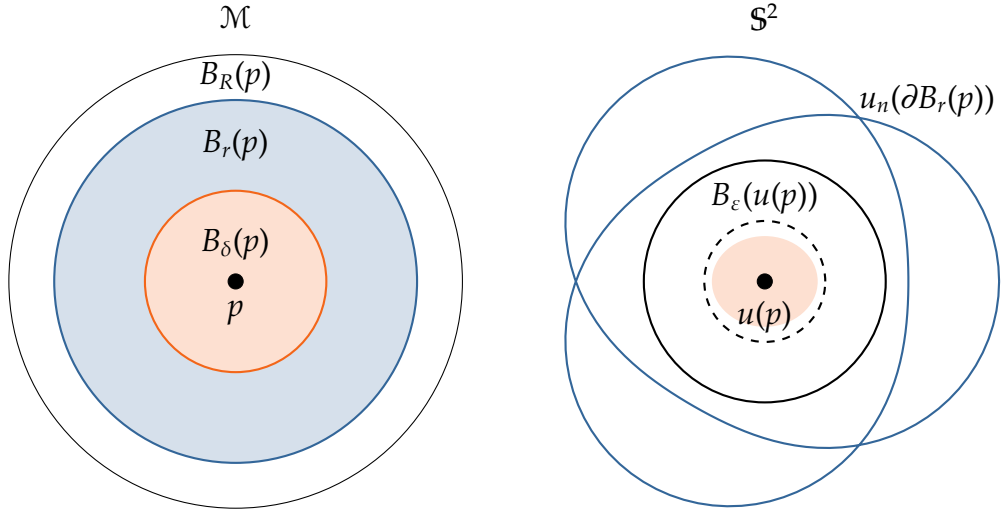

  \[
    \btkz
    \draw (0,0) circle (3.0);
    \draw[thick, color=warmblue, fill=warmblue!20] (0,0) circle (2.4);
    \draw[thick, color=ocre, fill=ocre!20] (0,0) circle (1.2);
    \draw (0,0) node[circle,draw=black, fill = black,inner sep=2pt, label=below:$p$] {};
    \node at (0,0.5) {$B_\delta(p)$};
    \node at (0,1.8) {$B_r(p)$};
    \node at (0,2.7) {$B_R(p)$};
    \node at (0,3.5) {$\M$};

    \begin{scope}[shift={(7,0)}]
      \draw[thick,color = warmblue, domain=0:360,samples=200,variable=\t]
      plot ({(2.5 + 0.6*cos(3*\t))*cos(2*\t)},
      {(2.5 + 0.6*cos(3*\t))*sin(2*\t)});
      \fill[color=ocre!20] (0,0) ellipse (0.7 and 0.6);
      \node at (2.15,2.4) {$u_n(\dB_r(p))$};
      \draw (0,0) node[circle,draw=black, fill = black,inner sep=2pt, label=below:$u(p)$] {};
      \draw[thick, color=black] (0,0) circle (1.6);
      \draw[thick,dashed, color=black] (0,0) circle (0.8);
      \node at (0,1.1) {$B_\e(u(p))$};
      \node at (0,3.5) {$\Sph$};
    \end{scope}
    \etkz
  \]
  \caption{
    Left: geodesic discs in $\M$, centered at $p$ and of radii $\delta < r < R$.
    Right: the codomain. The blue non-simple loop is $u_n(\dB_r(p))$, i.e., the image under $u_n$ of the geodesic circle of radius $r$ centered at $p$. The Solid black loop is a geodesic circle of radius $\e$ centered at $u(p)$. The dashed loop  is a geodesic circle of radius $\e/2$ centered at $u(p)$. The ochre domain is  $u_n(B_\delta(p))$.
  }
  \label{fig:1}
\end{figure}

%%%%%%%%%%%%%
\begin{proof}
  The statement of the lemma is illustrated in \figref{fig:1}.
  \begin{enumerate}[itemsep=0pt,label=(\alph*)]
    \item
          The assumption on $r$ implies that $u(p)\notin u(\dB_r(p))$. Since $u$ is continuous, there exists an $\e>0$, such that  $B_{2\e}(u(p))\cap u(\dB_r(p))=\emptyset$.
          By the uniform convergence $u_n\to u$, there exists an $N_1\in\bbN$ such that
          \[
            B_{\e}(u(p))\cap u_n(\dB_r(p))=\emptyset
          \]
          for all $n>N_1$.

    \item
          By the continuity of $u$, there is exists a $\delta\in(0,r)$, such that $u(B_{\delta}(p))\subset B_{\e/3}(u(p))$. By the uniform convergence $u_n\to u$, there exists an $N_2\in\bbN$ such that
          \[
            u_n(B_{\delta}(p))\subset B_{\e/2}(u(p))
          \]
          for all $n>N_2$,

    \item
          By the uniform convergence $u_n\to u$, there exists an $N_3\in\bbN$ such that the affine homotopy between $u(\dB_r(p))$ and $u_n(\dB_r(p))$ does not intersect $u(p)$ for every $n>N_3$. The assertion follows then from the stability of the degree under homotopies.

    \item
          Denote
          \[
            C = \overline{B_r(p)}\setminus B_{\delta}(p).
          \]
          Since its image $u(C)$ is compact and since by the definition of $r$, $u(p)\not\in u(C)$,
          there exists an $\e_1>0$ such that
          \[
            u(C)\cap B_{\e_1}(u(p))=\emptyset.
          \]
          By uniform convergence, there exists an $N_4\in\bbN$ such that
          \[
            u(p)\notin u_n(C)
          \]
          for all $n>N_4$.
  \end{enumerate}
  Finally, take $N_0=\max\{N_1,N_2,N_3,N_4\}$.
\end{proof}
%%%%%%%%%%%%%

With this lemma at hand, we can prove the following stability of the index:

%%%%%%%%%%
\begin{proposition}
  \label{prop:unif conv preserves local deg}
  Let $B_R(p)$ be a geodesic disc in $\M$.
  Let $u_n \in C(B_R(p); \Sph)$  be a sequence of local homeomorphisms converging uniformly to $u\in C(B_R(p);\Sph)$.
  Assume that $\ind_p(u)$ is well-defined.
  Then either all $u_n$ are eventually orientation-preserving, and  $\ind_p(u)=1$, or all $u_n$ are eventually orientation-reversing, and  $\ind_p(u)=-1$.
\end{proposition}
%%%%%%%%%%

%%%%%%%%%
\begin{proof}
  Recall from the proof of the first part of \thmref{thm:main} that for a local homeomorphism, the index of every point is well-defined, and equals to $+1$ if the homeomorphism is orientation preserving, and $-1$ if it is orientation reversing. Since the sign is locally constant, the sets of points of index $1$ and index $-1$ are both open, so by connectedness the sign is constant.

  Let $\e>0$, $\delta\in(0,r)$ and $N_0\in\bbN$, be as in \lemref{lem:unif_conv_k}, and fix some arbitrary $n>N_0$.
  As in \lemref{lem:unif_conv_k}, let $d=\ind_p(u)$.
  Denote
  \[
    V = B_{\e}(u(p)).
  \]
  Let $y\in V$. By Property (a) of \lemref{lem:unif_conv_k}, $y$ and $u(p)$ belong to the same connected component of
  $\Sph\setminus u_n(\dB_r(p))$, hence
  \beq\label{eq:prop_aux_1}
  \deg(u_n,B_r(p),y) = \deg(u_n,B_r(p),u(p)) = d,
  \eeq
  where the last equality follows from Property (c).
  Next, we note that \cite[Theorem 2.9(1)]{FG95} is applicable for local homeomorphisms between manifolds, and implies that
  \beq\label{eq:degree sum}
  \deg(u_n,B_r(p),y) = \sum_{x\in u_n^{-1}(y)\cap B_r(p)} \ind_x(u_n) = \pm |u_n^{-1}(y)\cap B_r(p)|,
  \eeq
  where the last equality follows from the fact that the index is either $+1$ or $-1$ for all $x$, as discussed above.

  Let $U = u_n^{-1}(V)$ and denote $u_U = u_n|_U$. Then, $u_U:U\to V$ is a local diffeomorphism (a fortiori a local homeomorphism) satisfying that for every $y\in V$, the fiber $u_U^{-1}(y)$ has cardinality $|d|$. It is therefore a covering map\footnote{Indeed, let $y\in V$ and let $\{x_1,\dots,x_d\} = u_n^{-1}(y)$. Take disjoint neighborhood $U_i$ of the $x_i$, on each of which $u_n$ is a homeomorphism, and shrink them such that their they have a common image $W$, which is a neighborhood of $y$. Each fiber of $W$ already contains one point in every $U_i$, hence $u_n^{-1}(W) = \bigsqcup_i U_i$, which proves the claim.}. The map $u_n$ restricted to a connected component of $U$ is still a covering map (see \cite{395568}), and since $V$ is simply connected, this restriction must be a homeomorphism (as follows from either \cite[Prop. 1.32]{Hat02}, or from the universal property of the universal cover). Thus $U$ has $|d|$ connected components $U_1,...,U_{|d|}$, and $u_n|_{U_i}$ is a homeomorphism for each $i$.

  Consider the fiber $u_U^{-1}(u(p))$,
  which consists of $|d|$ points,
  \[
    (x_1,\dots,x_d) \in U_1\times \dots \times U_{|d|}.
  \]
  By Properties (d) and (b) in  \lemref{lem:unif_conv_k},
  \[
    u_U^{-1}(u(p)) \subset B_\delta(p) \subset U,
  \]
  and since $B_\delta(p)$ is connected, it must be contained in a single connected component of $U$. Since the $|d|$ points $\{x_i\}$ belong at the same time to separate connected components of $U$ and to a single connected component of $U$, we conclude that $|d|=1$.

  It now follows from \eqref{eq:prop_aux_1} that for all $n>N_0$ all $u_n$ are either orientation-preserving or that all of them are orientation-reversing, and that $d=1$ or $d=-1$, accordingly.
\end{proof}

By combining Proposition~\ref{prop:unif conv} and Proposition~\ref{prop:unif conv preserves local deg}, we complete the proof of Theorem~\ref{thm:main}:

\begin{proof1}{of Theorem~\ref{thm:main} part 2}
  Let $f\in\Wiso$ be a map  whose Gauss map $\n$ satisfies $\ind_p(\n) \ne -1$ for some $p\in \M$.
  Assume by contradiction that there exists a sequence of immersions $(f_n)\subset\Wiso\cap C^2(\M;\R^3)$ converging to $f$ in $W^{2,2}$, and denote by $\n_n$ the corresponding Gauss maps.
  Consider some closed geodesic ball $\overline{B_R(p)}\subset \M$. By \propref{prop:unif conv}, $\n_n\to\n$ uniformly on $\overline{B_R(p)}$.
  For every $n$, $\n_n$ is continuously differentiable with a negative Jacobian, so by the inverse mapping theorem $\n_n$ is a local homeomorphism, with $\ind_x(\n_n)=-1$ for every $x\in B_R(p)$.
  Thus we can apply \propref{prop:unif conv preserves local deg} and obtain a contradiction to $\ind_p(\n) \ne -1$.
\end{proof1}

%%%%%%%%%%%%%%%%%%%%%%%%%%%%%%%%%%%%%%%%%%%%
\section{Constructing isometric immersions with branch points}
\label{sec:construction}

In this section we prove \thmref{thm:generic}, which
asserts that in a hyperbolic surface, every point has a neighborhood that can be immersed in $\R^3$ with a branch point at that point.
In particular, together with \thmref{thm:main}, it implies that, generally, $C^2$-isometric immersions are not dense in the space of $W^{2,2}$-isometric immersions.

%%%%%%%%%%%%%%%%%%%%%%%%%%%%%%%%%%%%%%%%%%%%
\subsection{Immersions of geodesic wedges}

The following proposition provides the the building block for the construction of the isometric immersion in the proof of \thmref{thm:generic}:

\begin{proposition}\label{prop:wedge immersion}
  Let $(\M,\g)$ be a two-dimensional Riemannian manifold. Let $p\in M$ be such that there is a neighborhood of $p$ in which $K_\g<0$. Let $\Gamma_1,\Gamma_2$ be two geodesics passing through $p$. Then there exists a neighborhood $U$ of $p$, and a $C^2$ isometric immersion $f:U\to\R^3$ such that $\Gamma_1,\Gamma_2$ are mapped into straight lines. In particular, any of the $4$ wedges of $U$ bounded by $\Gamma_1,\Gamma_2$ may be immersed isometrically in $\R^3$ with the same straight lines constraint.
\end{proposition}

\begin{proof}
  Choose local coordinates $(u,v)$ around $p$, such that $\Gamma_1$ is parametrized by $(u,0)$ and $\Gamma_2$ is parametrized by $(u,u)$. By the fundamental theorem of surface theory, the existence of a $C^2$ isometric immersion is equivalent to the existence of a continuous matrix-valued function
  \[
    \II=\brk{\begin{matrix}
        L & M \\
        M & N
      \end{matrix}},
  \]
  which together with the matrix representation of the metric $\g$ satisfies the Gauss-Codazzi equations. Assuming that $N$ (resp., $L$) does not vanish, the Gauss equation $\det \II = K_\g \det \g$ allows to express $L$ (resp. $N$) in terms of the two other entries. Then, the Codazzi equations form a quasi-linear hyperbolic system in two variables. As we shall see, the condition that $\Gamma_1$ and $\Gamma_2$ map into straight lines amounts to characteristic boundary conditions along those curves.

  We consider a representation of the Gauss-Codazzi equations due to Rozhdestvenskii and Poznyak \cite{Roz62}, in which the dependent variables are the pair of Riemann invariants (we use a slight variation which appears in \cite[p.~159]{PS96}). First, note that each of the two asymptotic directions (null-directions of $\II$) form a vector field; we call the two families of integral curves of these vector fields the \emph{characteristic curves} of the surface.
  The slopes in the $u$-$v$ plane of these characteristic curves (i.e., $dv/du$)  are given by
  \[
    r=\frac{-M-qk}{N}
    \Textand
    s=\frac{-M+qk}{N},
  \]
  where $k=\sqrt{-K_\g}$ and $q=\sqrt{\det\g}$.
  The inverse transformation is
  \beq
  L = \frac{2 r s q k}{s-r}
  \qquad
  M = -\frac{(s+r)q k}{s-r}
  \qquad
  N = \frac{2q k}{s-r}.
  \label{eq:LMN}
  \eeq

  Substituting $r,s$ into the Codazzi equations we obtain a system of equations of the form
  \beq\label{eq:RS}
  \begin{split}
    \partial_u r+s\partial_v r & =F((u,v),r,s)  \\
    \partial_u s+r\partial_v s & =F((u,v),s,r),
  \end{split}
  \eeq
  where $F$ is a smooth function determined by $\g$ and its first derivatives (see \cite[p.~159]{PS96} for details).

  Since $\Gamma_1,\Gamma_2$ are geodesics, they have vanishing geodesic curvature. Thus, the condition that they map to straight lines---vanishing total curvature---is equivalent to the condition that they have vanishing normal curvature, i.e., that $\Gamma_1$ and $\Gamma_2$ are characteristic curves. Since the variables $r$ and $s$ represent the slopes of the characteristic curves, this amounts to one of them being $0$ along $\Gamma_1$, and the other being $1$ along $\Gamma_2$. We choose
  \[
    \begin{aligned}
       & s=0
       & \qquad
       & \text{along $\Gamma_1$}  \\
       & r=1
       & \qquad
       & \text{along $\Gamma_2$}.
    \end{aligned}
  \]
  Note that we did not choose coordinates such that $\Gamma_2$ is parameterized by $(0,v)$ because then its slope would be undefined ("$r=\infty$"). Apart from avoiding $(0,v)$, the choice of parametrizing $\Gamma_2$ by $(u,u)$ is arbitrary.
  The proposition will thus be proved by obtaining $C^1$ solutions to \eqref{eq:RS} with these boundary conditions.
  The existence of solutions to such systems of quasilinear equations is well-known, and the method is outlined in \cite[Chapter~V, \S7]{CH61b}.
  For completeness, we describe below the main idea.

  Existence is proved in a quadrilateral domain with vertices $(h,0),(h,h),(-h,0),(-h,-h)$, for $h$ sufficiently small.
  We  start by rewriting \eqref{eq:RS} in integral form,
  \beq\label{eq:RS integral equations}
  \begin{split}
    r(u,v) & =1+\int_{u_{\eta}(u,v)}^u F(\eta_t(u,v),r(\eta_t(u,v)),s(\eta_t(u,v)))\,dt        \\
    s(u,v) & =\int_{u_{\gamma}(u,v)}^u F(\gamma_t(u,v),s(\gamma_t(u,v)),r(\gamma_t(u,v)))\,dt,
  \end{split}
  \eeq
  where $\eta_t(u,v)$ is the integral curve of the vector field $(1,s)$ starting at $(u,v)$ and $\gamma_t(u,v)$ is the integral curve of the vector field $(1,r)$ starting at $(u,v)$; the lower bounds of the integrals, $u_{\eta}(u,v)$ and $u_{\gamma}(u,v)$, are the values of $u$ at which the curves $\eta_t(u,v)$ and $\gamma_t(u,v)$ intersect the curves $\Gamma_2$ and $\Gamma_1$, respectively.

  The construction proceeds via Picard-like iterations, starting with $r^{(0)}(u,v)=1$ and $s^{(0)}(u,v)=0$.
  By choosing $h$ small enough, we guarantee that $u_\eta,u_\gamma$ are well defined, and that the iteration map is contractive in the uniform-convergence norm.  Thus, there exist continuous functions $r(u,v)$ and $s(u,v)$ satisfying \eqref{eq:RS integral equations}, whose directional derivatives $\partial_u+s\partial_v$ and $\partial_u+r\partial_v$ exist via \eqref{eq:RS}.

  It remains to show that $r$ and $s$ are differentiable. First we bound the second partial derivatives uniformly. Denote by $\Norm{(r,s)}_{C^2}$ the sum of the $\sup$-norms of all the second partial derivatives of $r$ and $s$.
  We compute the second partial derivatives explicitly by differentiating under the integral sign, once for the first derivatives and again for the second derivatives. This involves differentiating $u_\eta,u_\gamma$, which is done using the implicit function theorem, since the equation $\eta_t(u,v)-t=0$ defines $t$ as a function of $u,v$, which is precisely $u_\eta$ (and similarly with $u_\gamma$). We then get a bound of the form
  \[
    \|(r^{(n+1)},s^{(n+1)})\|_{C^2}\le g(h\|(r^{(n)},s^{(n)})\|_{C^2}),
  \]
  for some function $g:\R_+\to\R_+$ such that $g_0:=\lim_{x\to0^+}g(x)$ exists.
  Again, by choosing $h$ small enough we guarantee that $\|(r^{(n)},s^{(n)})\|_{C^2}$ is uniformly bounded, for example, by $g_0+1$. Finally we invoke the Arzela-Ascoli theorem, to obtain that a subsequence of the first derivatives of $(r^{(n)},s^{(n)})$ converges uniformly, which implies the differentiability of $r$ and $s$.
\end{proof}

Let $(\M,\g)$, $p\in\M$, and $\Gamma_1,\Gamma_2$ be as in \propref{prop:wedge immersion}, and consider one of the wedges $V$ bounded by $\Gamma_1,\Gamma_2$ (which are viewed  from here on as rays emanating from $p$). By the proposition, and after applying an appropriate translation and rotation, there exists an isometric immersion $f:V\to\R^3$ mapping $\Gamma_1,\Gamma_2$ into straight rays $l_1,l_2$ emanating from the origin in the $xy$-plane.
Note that this procedure defines $f$ uniquely up to reflection across the $xy$-plane.
Furthermore, we have a choice of orientation, i.e., whether $\n(p)$ points upwards or downwards; in this case, we will always choose $\n(p) = (0,0,1)$.

Let $u_1,u_2\in T_p\M$ be the unit vectors pointing in the direction of $\Gamma_1,\Gamma_2$. namely, $l_1 = df_p(u_1)$ and $l_2 = df_p(u_2)$.
$\n$ is perpendicular to $l_1$ along $\Gamma_1$, so it rotates in a fixed plane along $\Gamma_1$, see Figure~\ref{fig:wedge}.

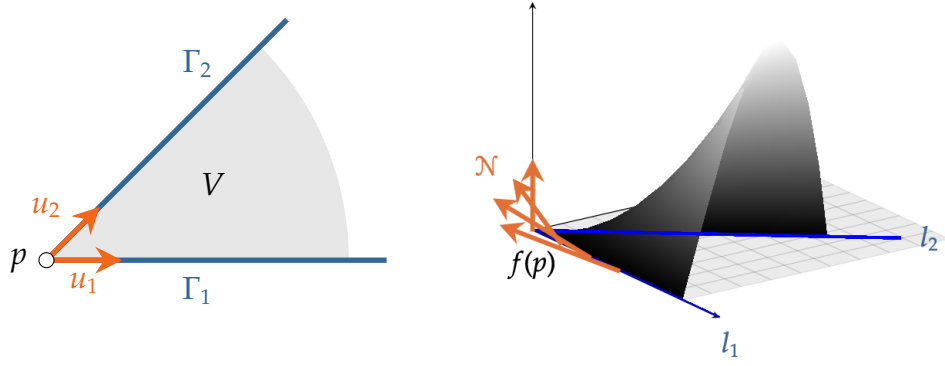
\begin{figure}
  \begin{center}
    \btkz

    % shaded wedge
    \fill[gray!20]
    (0,0) -- (4,0)
    arc[start angle=0,end angle=45,radius=4]
    -- cycle;
    % rays
    \draw[line width = 2, warmblue] (0,0) -- (4.5,0);
    \draw[line width = 2, warmblue] (0,0) -- ({4.5*cos(45)},{4.5*sin(45)});
    \node[color=warmblue]  at (2,-0.4) {$\Gamma_1$};
    \node[color=warmblue]  at (2,2.6) {$\Gamma_2$};
    \node at (2.2,1.) {$V$};
    \draw[-{Stealth[length=10pt,width=10pt]}, ocre, line width = 2] (0,0) -- (1,0);
    \draw[-{Stealth[length=10pt,width=10pt]}, ocre, line width = 2] (0,0) -- (0.707,0.707);
    \node[color = ocre] at (0.5,-0.3) {$u_1$};
    \node[color = ocre] at (0.,0.7) {$u_2$};
    \draw (0,0) node[circle,draw=black, fill = white,inner sep=2pt, label=left:$p$] {};

    \node[inner sep=0pt] at (9,1)
    {\includegraphics[width=7cm]{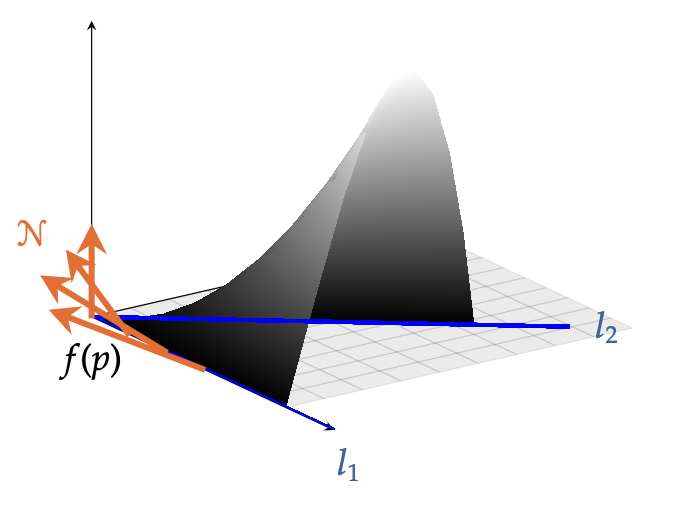}};
    \etkz
  \end{center}
  \caption{
    An isometric immersion of a wedge.
    Left: The domain $V\subset\M$ bounded  by the geodesic rays $\Gamma_1,\Gamma_2$, which intersect at $p$.
    Right: The surface $f(V)$, which $f(\Gamma_1)$ and $f(\Gamma_2)$ mapped along the rays $l_1$ and $l_2$.
    The orange arrows depict the rotation of the Gauss map along $l_1$, which is in this case an \emph{outward immersion}.
  }
  \label{fig:wedge}
\end{figure}

Together with $d\n_p(u_1)\perp\n_p$, we get that $d\n_p(u_1)$ is in the $xy$ plane, and is perpendicular to $l_1$.
It is either contained in the half plane (bounded by $l_1$) containing the $xy$-projection of $f(V)$, in which case $0<\ip{d\n_p(u_1),df_p(u_2)}=-II(u_1,u_2)$, and we say it points inwards, or it is contained in the other half plane, in which case $0>-II(u_1,u_2)$ and it is said to point outwards.
Everything holds analogously for $d\n_p(u_2)$. By the symmetry of the second fundamental form, $II(u_1,u_2)=II(u_2,u_1)$, so either $d\n_p(u_1),d\n_p(u_2)$ both point inwards or they both point outwards. In the first case we say $f$ is an \emph{inward immersion} of $V$, and in the second an \emph{outward immersion} of $V$ (see Figure~\ref{fig:wedge}). Note that if $f$ is inward, its reflection across the $xy$-plane is outward.
In the following, we will denote by $f^+$ the inward immersions and by $f^-$ the corresponding outward immersion.

The next lemma will be used in \secref{sec:gluing} to calculate the local degree of the Gauss map of the constructed immersion.

\begin{lemma}\label{lem:wedge rotation}
  Let $f:U\to \M$ be constructed as above (either an inward or outward immersion when restricted to $V$).
  Let $\gamma_r:[0,1]\to\M$ be an oriented geodesic circle of radius $r$ centered at $p$. Denote by $0<\alpha<\pi$ the angle between $\Gamma_1$ and $\Gamma_2$. Define $\tilde{\n}:\M\to\R^2$ to be the projection of $\n$ onto the $xy$ plane.
  Then for sufficiently small $r$, as $\gamma_r$ traverses $V$, the curve $\sigma_r:=\tilde{\n}\circ\gamma_r:[0,1]\to\R^2$ sweeps out an angle of $\alpha-\pi$ with respect to the origin.
\end{lemma}

\begin{proof}
  For every $w\in\R^2\setminus\{0\}\cong \bbC\setminus\{0\}$, we view $\arg w$ as an element in the additive group $\R/(2\pi\bbZ)$. For $i=1,2$ denote by $q_i=\gamma_r(t_i)$ the intersection of $\gamma_r$ with $\Gamma_i$. For concreteness, we label rays such that one moves from $l_1$ to $l_2$ counterclockwise,  i.e., that $\arg l_2-\arg l_1=\alpha$ (and not $-\alpha$).

  Recall that $\n$ is perpendicular to $l_i$ along $\Gamma_i$.
  Since $l_i$, viewed as an element of $\bbS^2$ equals its own projection onto the $xy$-plane, the orthogonality of $\n$ and $l_i$ is preserved by the projection, which implies that $\sigma_r(t_i)\perp l_i$. By the preceding discussion,  for $r$ sufficiently small, the half plane in which $\sigma_r(t_i)$ lies is determined by whether $f$ is inward or outward, specifically,
  \[
    \arg\sigma_r(t_1)=\arg l_1+(-1)^k\frac{\pi}{2}
    \Textand
    \arg\sigma_r(t_2)=\arg l_2-(-1)^k\frac{\pi}{2},
  \]
  where $k=0$ if $f$ is inward and $k=1$ if $f$ is outward; see the following figure for $k=1$.

  \begin{center}
    \btkz
    % shaded wedge
    \fill[gray!20]
    (0,0) -- (4,0)
    arc[start angle=0,end angle=45,radius=4]
    -- cycle;
    % rays
    \draw[line width = 2, warmblue] (0,0) -- (4.5,0);
    \draw[line width = 2, warmblue] (0,0) -- ({4.5*cos(45)},{4.5*sin(45)});
    \node[color=warmblue]  at (2.6,-0.4) {$\Gamma_1$};
    \node[color=warmblue]  at (2,2.6) {$\Gamma_2$};
    \node at (2.8,1.) {$V$};
    \draw (0,0) node[circle,draw=black, fill = white,inner sep=2pt, label=left:$p$] {};
    \draw[ocre]  (1.5,0) arc[start angle=0,end angle=45,radius=1.5];

    \draw (1.5,0) node[circle,draw=black, fill = white,inner sep=2pt, label=below:$q_1$] {};
    \draw (1.06,1.06) node[circle,draw=black, fill = white,inner sep=2pt, label=above left:$q_2$] {};
    \node[ocre] at (1.8,0.6) {$\gamma_r$};
    \node[warmblue] at (1,0.5) {$\alpha$};

    \draw[-{Stealth[length=10pt,width=10pt]}, warmblue, line width = 1] (5,0.5) -- (6,0.5);
    \node[warmblue] at (5.4,0.85) {$\tilde{\n}$};

    \begin{scope}[xshift=9cm, yshift=1.5cm]
      % Axes
      \draw[->] (-2,0) -- (2,0) node[below] {$x$};
      \draw[->] (0,-2) -- (0,2) node[left] {$y$};
      \draw[dashed] ({-1.5*cos(-15)},{-1.5*sin(-15)}) -- ({1.5*cos(-15)},{1.5*sin(-15)});
      \draw[dashed] ({-1.5*cos(30)},{-1.5*sin(30)}) -- ({1.5*cos(30)},{1.5*sin(30)});
      \draw ({1.5*cos(-15)},{1.5*sin(-15)}) node[circle,draw=black, fill = white,inner sep=2pt, label=below:$l_1$] {};
      \draw ({1.5*cos(30)},{1.5*sin(30)}) node[circle,draw=black, fill = white,inner sep=2pt, label=above:$l_2$] {};
      \draw[warmblue]  ({1.5*cos(-15)},{1.5*sin(-15)}) arc[start angle=-15,end angle=30,radius=1.5];
      \node[warmblue] at (1.1,0.1) {$\alpha$};

      \draw[thick, dotted] (0,0) -- ({2*cos(120)},{2*sin(120)});
      \draw ({1.8*cos(120)},{1.8*sin(120)}) node[circle,draw=black, fill = white,inner sep=2pt, label=left:$\sigma_r(t_2)$] {};
      \draw[thick, dotted] (0,0) -- ({2*cos(-105)},{2*sin(-105)});
      \draw ({1.2*cos(-105)},{1.2*sin(-105)}) node[circle,draw=black, fill = white,inner sep=2pt, label=left:$\sigma_r(t_1)$] {};
      \draw[rotate=30] (0.3,0) -- (0.3,0.3) -- (0,0.3);
      \draw[rotate=-105] (0.3,0) -- (0.3,0.3) -- (0,0.3);
      \draw[ocre]  ({1.*cos(120)},{1.*sin(120)}) arc[start angle=120,end angle=255,radius=1.];
      \node[ocre] at (-1.3,-0.25) {$\sigma_r$};

    \end{scope}

    \etkz
  \end{center}

  If we prove that from $t_1$ to $t_2$, $\sigma_r(t)$ goes clockwise (i.e., sweeps a negative angle) and does not complete a full turn, then the angle that $\sigma_r$ sweeps out is equal to the representative of
  \[
    \arg l_2-(-1)^k\frac{\pi}{2}-\brk{\arg l_1+(-1)^k\frac{\pi}{2}}
  \]
  in $(-2\pi,0]$, which is equal to $\alpha-\pi$ independently of $k$.

  We start with showing that $\sigma_r$ rotates clockwise. It suffices to prove that for sufficiently small $r$, $\det(\sigma_r(t),\sigma_r'(t))<0$ for every $t$, where we view $(w_1,w_2)$ as a $2\times2$ matrix for $w_1,w_2\in\R^2$. We choose oriented normal coordinates centered at $p$, view $\tilde{\n}$ as a map from $\R^2$ to $\R^2$ and $\gamma$ as a curve in $\R^2$ (choosing normal coordinates is just for convenience, so that we have $\gamma_r(t)=r\gamma_1(t)$, and we assume that $\gamma_1(t)$ is well-defined just to simplify notations).
  \[
    \begin{split}
       & \sgn\det(\sigma_r(t),\sigma_r'(t))                                                            \\
       & \qquad =\sgn\det\brk{d\tilde{\n}_0(\gamma_r(t))+o(r),d\tilde{\n}_{\gamma_r(t)}(\gamma_r'(t))} \\
       & \qquad =\sgn\det(d\tilde{\n}_0)\cdot\sgn\det\brk{\gamma_r(t)+d\tilde{\n}_0^{-1}(o(r)),
      \,\gamma_r'(t)+d\tilde{\n}_0^{-1}d\tilde{\n}_{\gamma_r(t)}-d\tilde{\n}_0)(\gamma_r'(t))}         \\
       & \qquad =-\sgn\brk{r^2\det(\gamma_1(t)+d\tilde{\n}_0^{-1}(o(1)),
      \,\gamma_1'(t)+d\tilde{\n}_0^{-1}(d\tilde{\n}_{\gamma_r(t)}-d\tilde{\n}_0))(\gamma_1'(t))}       \\
       & \qquad \underset{r\to0}{\to}-\sgn\det(\gamma_1(t),\gamma_1'(t)) = -1,
    \end{split}
  \]
  where we used the continuity of $d\tilde{\n}$ at zero, and the fact that $\det d\tilde{\n}(0)<0$ by Gauss' theorem.

  The fact that $\sigma_r(t)$ does not complete a full turn from $t_i$ to $t_{i+1}$ follows from the fact that then the winding number of the full $\sigma_r$ (when we do not restrict it to $V$) would be $<-1$, since we saw $\arg\sigma_r(t)$ is strictly decreasing (this is the meaning of $\sigma_r$ rotating clockwise), but on the other hand it must equal $-1$ since $\tilde{\n}$ is an orientation reversing homeomorphism near $0$.
\end{proof}

%%%%%%%%%%%%%%%%%%%%%%%%%%%%%%%%%%%%%%%%%%%%
\subsection{Gluing the wedges}
\label{sec:gluing}

In this section, we utilize the construction in the proof of \propref{prop:wedge immersion} to construct $W^{2,\infty}$ isometric immersions with branch points:

\begin{theorem}\label{thm:general branch point}
  Let $(\M,\g)$ and $p\in \M$ be as in \propref{prop:wedge immersion}, and let $2\le m\in\bbN$. Then there exists a neighborhood $U$ of $p$ and a $W^{2,\infty}$ ($=C^{1,1}$) isometric immersion $f:U\to\R^3$ such that $\ind_p(\n)=1-m$, where $\n$ is the Gauss map. In particular, if $m\ge3$ then $p$ is a branch point of order $m$.
\end{theorem}

\begin{proof}
  We divide the proof into steps:
  \begin{enumerate}[itemsep=0pt,label=(\alph*)]
    \item Construction of $f$.
    \item Proving that $\n$ is everywhere defined and continuous.
    \item Proving that $f$ is in $W^{2,\infty}$.
    \item Proving that $p$ is a branch point of order $m$.
  \end{enumerate}

  \paragraph{Step (a)}
  Choose $2m\ge4$ geodesic rays $\Gamma_1,...,\Gamma_{2m}$ emanating from $p$, ordered counter-clockwise by their indices, such that $\alpha_i<\pi$ for every $i$, where $\alpha_i$ is the angle between $\Gamma_i$ and $\Gamma_{i+1}$.
  Clearly, $\sum_{i=1}^{2m} \alpha_i = 2\pi$.
  Correspondingly, denote by $l_1,...,l_{2m}$ rays in the $xy$-plane in $\R^3$, emanating from the origin, ordered counter-clockwise by their indices, with an angle $\alpha_i$ between $l_i$ and $l_{i+1}$.

  By \propref{prop:wedge immersion} and the discussion following it, there exists for each $i=1,...,2m$ a geodesic sector $V_i$ bounded by $\Gamma_i$ and $\Gamma_{i+1}$, and isometric immersions $g_i^{\pm}:V_i\to\R^3$ mapping $\Gamma_i$ and $\Gamma_{i+1}$ into $l_i$ and $l_{i+1}$ respectively, where $g_i^+$ are inward immersions and $g_i^-$ outward immersions.
  We choose $f_i = g_i^+$ for odd $i$, and $f_i = g_i^{-}$ for even $i$.
  We define $f$ by gluing all the $f_i$'s, i.e., $f\vert_{V_i}=f_i$ for all $i$.
  The map $f$ is continuous by construction.
  The choice of alternating between inward and outward immersions is required for the continuity of $\n$, and the $W^{2,\infty}$ regularity of $f$.

  \paragraph{Step (b)}
  For every $i$, the Gauss map $\n_i$ of $f_i$ is continuous up to the boundary, hence it remains to show that $\n_i,\n_{i+1}$ agree at the transition between the sectors. The argument is a particular case of the Beltrami--Enneper theorem \cite[Chapter 4, Theorem 7]{Spi99}, in the case where the asymptotic curves map into straight lines.

  Fix some $i\in\{1,...,2m\}$ and let ${X,Y}$ be an orthonormal frame for $TM$ along $\Gamma_i$, such that $X$ is parallel to $\Gamma_i$.
  By the definition of the second fundamental form, along $\Gamma_i$,
  \[
    d\n_i(X) = - \II_i(X,X) \, df_i(X) -  \II_i(X,Y) \, df_i(Y).
  \] es because $X$ is an asymptotic direction, whereas by the Gauss equation
  \[
    \II_i(X,X)\II_i(Y,Y) - (\II_i(X,Y))^2 = K_\g.
  \]
  Thus,
  \[
    d\n_i(X) = \pm \sqrt{-K_\g}\, df_i(Y).
  \]
  In particular, the rates of rotation of $\n_i,\n_{i-1}$ along $\Gamma_i$ either coincide or are opposite to each other. Since we chose the sectors to alternate between inward and outward immersions, the rates of rotation coincide. Therefore $\n_i,\n_{i-1}$ agree along $\Gamma_i$, so $\n$ is defined and continuous everywhere.

  \paragraph{Step (c)}
  $f$ is uniformly bounded in $C^2$ in the interior of each sector by construction, and thus, in order to show that $f\in W^{2,\infty}$, it remains to show that the derivative of $f$ is continuous.

  We only need to verify the continuity of $df$ along the sectors' boundaries. Fix some $i$ and let $q\in\Gamma_i$ ($q$ may be the vertex $p$). Choose coordinates around $q$ such that $\partial_1(q)$ is in the direction of $\Gamma_i$ and $\partial_2\perp\partial_1$. $\partial_1f_i(q)=\partial_1f_{i-1}(q)$, because $f_i=f_{i-1}$ along $\Gamma_i$. By the continuity of the normal demonstrated in Step (b), both $\partial_2f_i(q)$ and $\partial_2f_{i-1}(q)$ are either $\n(q)\times\partial_1f_i(q)$ or $\partial_1f_i(q)\times\n(q)$, depending on orientation. From this we conclude that $df_{i-1}\vert_q=df_i\vert_q$. This implies that $df$ is defined and continuous at every point in $\Gamma_i\setminus\{p\}$. As for $p$,  we get inductively that $df_i\vert_p=df_1\vert_p$ for every $i$, so $df$ is defined and continuous at $p$.

  \paragraph{Step (d)}
  We prove that the local degree of the Gauss map at $f(p)$ is $1-m$, hence $f(p)$ is a branch point if $m>2$. Define $\tilde{\n}$, $\sigma_r$ like in \lemref{lem:wedge rotation}. We find $\deg_p\n$ by calculating the winding number of $\sigma_r$ around $0$, for $r$ sufficiently small.
  The winding number is the sum of the angles swept out by $\sigma_r$ in each sector, divided by $2\pi$, which by \lemref{lem:wedge rotation} is equal to
  \[
    \frac{1}{2\pi}\sum_{i=1}^{2m}(\alpha_i-\pi)=1-m.
  \]
\end{proof}

%%%%%%%%%%%%%%%%%%%%%%%%%%%%%%%%%%%%%%%%%%%%%
\addcontentsline{toc}{section}{References}
\begingroup
\footnotesize
\bibliographystyle{amsalpha}
\bibliography{arXiv_vr1.bbl}
\endgroup

\end{document}